\documentclass[12pt]{article}
\usepackage{amsthm}
\usepackage{amsmath}
\usepackage{amssymb,bm}
\usepackage{graphicx}
\usepackage{hyperref}
\usepackage{xcolor}
\usepackage{comment}
\usepackage[top=1.05in,left=1.05in,right=1.05in,bottom=1.05in]{geometry}
\usepackage{times} 
\usepackage{BOONDOX-cal}
\numberwithin{equation}{section} 

\newtheorem{Thm}{Theorem}
\newtheorem{Lem}[Thm]{Lemma}
\newtheorem{Def}[Thm]{Definition}
\newtheorem{Pro}[Thm]{Proposition}

\newtheorem{Prob}[Thm]{Problem}

\newtheorem{Note}[Thm]{Notation}

\newcommand{\bb}[1]{\mathbb{ #1}}

\bmdefine\Bone{1}

\newcommand{\bra}[1]{\overline{#1}}

\newcommand{\hf}{\displaystyle\frac{1}{2}}
\newcommand{\nth}[1]{\displaystyle\frac{1}{#1}}

\newcommand{\Tld}[1]{\widetilde{#1}}

\def\XXint#1#2#3{{\setbox0=\hbox{$#1{#2#3}{\int}$ }
\vcenter{\hbox{$#2#3$ }}\kern-.6\wd0}}

\newcommand{\rhs}{right-hand side}

\newcommand{\Ga}{\alpha}

\newcommand{\Gl}{\lambda}

\newcommand{\Gs}{\sigma}

\newcommand{\Gz}{\zeta}

\bmdefine\BGa{\alpha}
\bmdefine\BGb{\beta}
\bmdefine\BGd{\delta}
\bmdefine\BGe{\epsilon}
\bmdefine\BGve{\varepsilon}
\bmdefine\BGf{\phi}
\bmdefine\BGvf{\varphi}
\bmdefine\BGg{\gamma}
\bmdefine\BGc{\chi}
\bmdefine\BGi{\iota}
\bmdefine\BGk{\kappa}
\bmdefine\BGl{\lambda}
\bmdefine\BGn{\eta}
\bmdefine\BGm{\mu}
\bmdefine\BGv{\nu}
\bmdefine\BGp{\pi}
\bmdefine\BGth{\theta}
\bmdefine\BGvth{\vartheta}
\bmdefine\BGr{\rho}
\bmdefine\BGvr{\varrho}
\bmdefine\BGs{\sigma}
\bmdefine\BGvs{\varsigma}
\bmdefine\BGt{\tau}
\bmdefine\BGj{\tau}
\bmdefine\BGu{\upsilon}
\bmdefine\BGo{\omega}
\bmdefine\BGx{\xi}
\bmdefine\BGy{\psi}
\bmdefine\BGz{\zeta}
\bmdefine\BGD{\Delta}
\bmdefine\BGF{\Phi}
\bmdefine\BGG{\Gamma}
\bmdefine\BGL{\Lambda}
\bmdefine\BGP{\Pi}
\bmdefine\BGT{\Theta}
\bmdefine\BGS{\Sigma}
\bmdefine\BGU{\Upsilon}
\bmdefine\BGO{\Omega}
\bmdefine\BGX{\Xi}
\bmdefine\BGY{\Psi}

\bmdefine\BFM{\mathfrak{M}}
\bmdefine\BFb{\mathfrak{b}}
\bmdefine\BFk{\mathfrak{k}}
\bmdefine\BFm{\mathfrak{m}}
\bmdefine\BFu{\mathfrak{u}}
\bmdefine\BFv{\mathfrak{v}}

\newcommand{\CM}{{\mathcal M}}

\newcommand{\CS}{{\mathcal S}}

\newcommand{\CW}{{\mathcal W}}

\bmdefine\BCA{{\mathcal A}}
\bmdefine\BCB{{\mathcal B}}
\bmdefine\BCC{{\mathcal C}}
\bmdefine\BCD{{\mathcal D}}
\bmdefine\BCE{{\mathcal E}}
\bmdefine\BCF{{\mathcal F}}
\bmdefine\BCG{{\mathcal G}}
\bmdefine\BCH{{\mathcal H}}
\bmdefine\BCI{{\mathcal I}}
\bmdefine\BCJ{{\mathcal J}}
\bmdefine\BCK{{\mathcal K}}
\bmdefine\BCL{{\mathcal L}}
\bmdefine\BCM{{\mathcal M}}
\bmdefine\BCN{{\mathcal N}}
\bmdefine\BCO{{\mathcal O}}
\bmdefine\BCP{{\mathcal P}}
\bmdefine\BCQ{{\mathcal Q}}
\bmdefine\BCR{{\mathcal R}}
\bmdefine\BCS{{\mathcal S}}
\bmdefine\BCT{{\mathcal T}}
\bmdefine\BCU{{\mathcal U}}
\bmdefine\BCV{{\mathcal V}}
\bmdefine\BCW{{\mathcal W}}
\bmdefine\BCX{{\mathcal X}}
\bmdefine\BCY{{\mathcal Y}}
\bmdefine\BCZ{{\mathcal Z}}

\bmdefine\Bzr{ 0}
\bmdefine\Ba{ a}
\bmdefine\Bb{ b}
\bmdefine\Bc{ c}
\bmdefine\Bd{ d}
\bmdefine\Be{ e}
\bmdefine\Bf{ f}
\bmdefine\Bg{ g}
\bmdefine\Bh{ h}
\bmdefine\Bi{ i}
\bmdefine\Bj{ j}
\bmdefine\Bk{ k}
\bmdefine\Bl{ l}
\bmdefine\Bm{ m}
\bmdefine\Bn{ n}
\bmdefine\Bo{ o}
\bmdefine\Bp{ p}
\bmdefine\Bq{ q}
\bmdefine\Br{ r}
\bmdefine\Bs{ s}
\bmdefine\Bt{ t}
\bmdefine\Bu{ u}
\bmdefine\Bv{ v}
\bmdefine\Bw{ w}
\bmdefine\Bx{ x}
\bmdefine\By{ y}
\bmdefine\Bz{ z}
\bmdefine\BA{ A}
\bmdefine\BB{ B}
\bmdefine\BC{ C}
\bmdefine\BD{ D}
\bmdefine\BE{ E}
\bmdefine\BF{ F}
\bmdefine\BG{ G}
\bmdefine\BH{ H}
\bmdefine\BI{ I}
\bmdefine\BJ{ J}
\bmdefine\BK{ K}
\bmdefine\BL{ L}
\bmdefine\BM{ M}
\bmdefine\BN{ N}
\bmdefine\BO{ O}
\bmdefine\BP{ P}
\bmdefine\BQ{ Q}
\bmdefine\BR{ R}
\bmdefine\BS{ S}
\bmdefine\BT{ T}
\bmdefine\BU{ U}
\bmdefine\BV{ V}
\bmdefine\BW{ W}
\bmdefine\BX{ X}
\bmdefine\BY{ Y}
\bmdefine\BZ{ Z}

\begin{document}

\title{Complete characterization of symmetric Kubo-Ando operator means satisfying  Moln{\'a}r's weak associativity}
\date{\today}
\author{Yury Grabovsky\footnote{Email: yury@temple.edu}\\Temple University\\Department of Mathematics\\Philadelphia, PA, USA 19122\\\\ Graeme W.\ Milton\footnote{Email: graeme.milton@utah.edu}\\University of Utah\\Department of Mathematics\\Salt Lake City, UT, USA 84112\\\\ Aaron Welters\footnote{Email: awelters@fit.edu}\\Florida Institute of Technology\\Department of Mathematics and Systems Engineering\\Melbourne, FL, USA 32901}

\maketitle

\begin{abstract}
   We provide a complete characterization of a subclass of weakly associative means of positive operators in the class of symmetric Kubo-Ando means. This class, which includes the geometric mean, was first introduced and studied in L.\ Moln{\'a}r, \textit{Characterizations of certain means of positive operators}, Linear Algebra Appl. 567 (2019) 143-166, where
   he gives a characterization of this subclass (which we call the Moln{\'a}r class of means) in terms of the properties of their representing operator monotone functions. Moln\'ar's paper leaves open the problem of determining if the geometric mean is the only such mean in that subclass. Here we give a negative answer to this question 
   by constructing an order-preserving bijection between this class and a class of real measurable odd periodic functions bounded in absolute value by 1/2. Each member of the latter class defines a Moln{\'a}r mean by an explicit exponential-integral representation. From this we are able to understand the order structure of the Moln\'ar class and construct several infinite families of explicit examples of Moln{\'a}r means that are not the geometric mean. 
   Our analysis also shows how to modify Moln{\'a}r’s original characterization so that the geometric mean is the only one satisfying the requisite set of properties.
\end{abstract}


\sloppy

\section{Introduction}

In this paper, following \cite{19LM},  we adopt the following notations:
\begin{itemize}
\item $H$ always denotes a complex Hilbert space. 
\item $B(H)$ - the set of all bounded linear operators on $H$.
\item $B(H)^{+}$ - the set of all positive semidefinite operators in $B(H)$.
\item  $B(H)^{++}$ - the set of all invertible operators in $B(H)^{+}$.
\item $I$ denotes the identity operator on $H$.
\item The order relation $A\leq B$ for $A,B\in B(H)^{+}$ means $B-A\in B(H)^{+}$ [i.e., $\leq $ is the Loewner order on $B(H)^{+}$].
\item We write $A_n\downarrow A$ if $(A_n)$ is a monotonically decreasing sequence in $B(H)^{+}$, i.e., $A_1\geq A_2\geq \ldots$, and $A_n$ converges strongly to $A$ in $B(H)$. 
\end{itemize}

Following \cite{80KA} (see also \cite{10FH} and  \cite[Chap.\ 36 and 37]{19BS}), a binary operation 
\begin{gather*}
    \sigma:B(H)^{+}\times B(H)^{+}\rightarrow B(H)^{+},\;(A,B)\mapsto\sigma (A,B)=:A\sigma B
\end{gather*}
is called a \textit{connection}\footnote{The term comes from the study of connections of elements in an electrical network \cite{69AD,75AT}.} on $B(H)^{+}$ if the following requirements are fulfilled:
\begin{itemize}
    \item[(I)] (Joint Monotonicity) $A\leq C$ and $B\leq D$ imply $A\sigma B\leq C\sigma D$;
    \item[(II)] (Transformer Inequality) $C(A\sigma B)C\leq (CAC)\sigma (CBC)$;
    \item[(III)] (Upper Semicontinuity) $A_n\downarrow A$ and $B_n\downarrow B$ imply $A_n\sigma B_n\downarrow A\sigma B$.
\end{itemize}
Moreover, a connection $\sigma$ is called a \textit{mean} (or, more precisely, a \textit{Kubo-Ando mean}) or \textit{symmetric}, respectively, if
\begin{itemize}
    \item[(IV)] (Normalization) $I\sigma I=I$;
    \item[(V)] (Permutation Symmetry) $A\sigma B=B\sigma A,\; \forall A,B\in B(H)^{+}$.
\end{itemize}
In particular, a \textit{symmetric Kubo-Ando mean} is any binary operation $\sigma$ on $B(H)^{+}$ satisfying (I)-(V). 

Three important examples of symmetric Kubo-Ando means (which motivated the axiomatic approach of \cite{80KA} to positive operator means) are the \textit{arithmetic mean} $\nabla$, \textit{harmonic mean} $!$, and \textit{geometric mean} $\#$  (see also \cite{75PW, 01LL, 04ALM} and the recent survey \cite{24LL}) which, for all $A,B\in B(H)^{++}$, are given by:
\begin{gather}
    A\nabla B=\frac{1}{2}(A+B);\label{def:ArithmMean}\\
    A!B=2\left(A^{-1}+B^{-1}\right)^{-1};\label{def:HarmMean}\\
    A\#B=A^{1/2}(A^{-1/2}BA^{-1/2})^{1/2}A^{1/2}.\label{def:GeomMean}
\end{gather}
The values of these means at non-invertible elements of $B(H)^{+}$ can be determined from (III) by passing to the limit of monotonically decreasing sequences from $B(H)^{++}$ (see, e.g., \cite{80KA}). For instance, if $A,B\in B(H)^{+}$ then the formula for $A\nabla B$ is still (\ref{def:ArithmMean}), whereas if $A\in B(H)^{++}, B\in B(H)^{+}$ then the formula for $A\#B$ is still (\ref{def:GeomMean}). But in general, if $A,B\in B(H)^{+}$ then (\ref{def:HarmMean}) and (\ref{def:GeomMean}) become 
\begin{gather}
    A!B=2\lim_{n\rightarrow \infty}\left(A_n^{-1}+B_n^{-1}\right)^{-1},\label{def:HarmMeanLimitDef}\\
    A\#B=\lim_{n\rightarrow \infty}A_n^{1/2}(A_n^{-1/2}B_nA_n^{-1/2})^{1/2}A_n^{1/2},\notag
\end{gather}
for any two sequences $(A_n), (B_n)$ in $B(H)^{++}$ such that $A_n\downarrow A$ and $B_n\downarrow B$.  

Aside from the factor of $2$, the right hand side of \eqref{def:HarmMeanLimitDef} has been given the name  `parallel sum.' More precisely, the parallel sum is denoted by $:$ and gives another well-known example of a connection:
\begin{gather}
    A:B=\frac{1}{2}A!B, \text{ for } A,B\in B(H)^+.\label{def:ParallelSum}
\end{gather}
The concept of the parallel sum was introduced for positive semidefinite matrices in \cite{69AD} and extended to bounded (positive semidefinite) linear operators in \cite{75AT} (see also \cite{80KA, 76NA}).

As shown in \cite{80KA}, there is an important relationship between connections, Loewner's theory on operator-monotone functions, and the properties of a special class of analytic functions called Herglotz functions \cite{00GT, 01GKMT, 14GT} (also called Nevanlinna \cite{09ADL, 15AL}, Herglotz-Nevanlinna \cite{11ABT, 19LN, 22LO, 22OL}, Pick \cite{56AD, 64AD, 74WD, 08CB}, or $R$-functions \cite{51EW, 52EW, 54WN, 74KK}, \cite[Appendix]{77KN}). Let us briefly elaborate on this. 

First, a function $f$ is called a \textit{Herglotz function} if $f:\mathbb{C}^+\rightarrow \mathbb{C}^+\cup \mathbb{R}$ is analytic, where $\mathbb{C}^+=\{z\in \mathbb{C}:\operatorname{Im}z>0\}$ denotes the open upper half-plane. Next, let us introduce the following notation. 
\begin{Note}
Let $OM_+$ denote\footnote{The restrictions of functions in $OM_+\setminus\{0\}$ to the complex upper half-plane is denoted by $\CS([-\infty,0])$ in \cite[Appendix, Sec.\ 4]{77KN}. The functions in this class, and therefore, in $OM_+\setminus\{0\}$, have an exponential-integral representation \cite[Appendix, Sec.\ 5]{77KN} that will play an important role in our paper.} the class of all analytic functions $f:\mathbb{C}\setminus(-\infty,0]\rightarrow \mathbb{C}$ satisfying $f(x)\geq 0$ if $x\in \mathbb{R}$ with $x>0$ and $\operatorname{Im}f(z)\geq 0$ if $z\in \mathbb{C}$ with $\operatorname{Im} z>0$.     
\end{Note}

In particular, if $f\in OM_+$ then its restriction $f|_{\mathbb{C}^+}:\mathbb{C}^+\rightarrow \mathbb{C}^+\cup \mathbb{R}$ to $\mathbb{C}^+$ is a Herglotz function. Next, a function 
\begin{gather*}
f:(0,\infty)\rightarrow [0,\infty)    
\end{gather*}
is called a \textit{(positive) operator monotone function} if, for every Hilbert space $H$,
\begin{gather*}
    A,B\in B(H)^{++},\;A\leq B \Rightarrow f(A)\leq f(B),
\end{gather*}
where $f$ is defined on $B(H)^{+}$ by the functional calculus for self-adjoint operators on $H$  \cite{80RS, 14GT}. A deep result of Loewner \cite{34KL} (see also \cite{74WD, 19BS}) says that every positive operator monotone function $g$ has a unique analytic continuation to a function $f\in OM_+$ [or, equivalently, $f(x)=g(x)$ for all $x>0$] and conversely. \textit{Given this correspondence between positive operator monotone functions and the elements of $OM_+$, we will abuse notation throughout the rest of this paper and not distinguish them unless necessary.}

More precisely, the map $f\mapsto f|_{(0,\infty)}$ sending $f\in OM_+$ to its restriction $f|_{(0,\infty)}:(0,\infty)\rightarrow [0,\infty)$ is a bijection of $OM_+$ onto the class of positive operator monotone functions. Furthermore, the map $m\mapsto f$, defined by
\begin{gather}
    f(z)=a+bz+\int_{(0,\infty)}\frac{z(1+\lambda)}{z+\lambda}dm(\lambda),\;\text{for } z\in \mathbb{C}\setminus(-\infty,0],\label{OperMonoFuncIntegralRepr}
\end{gather}
where $a=m(\{0\})$ and $b=m(\{\infty\})$, establishes a bijection between the class of finite (positive) Borel measures on $[0,\infty]$ onto this class $OM_+$ of functions.

Now a result of \cite{80KA} says that positive operator monotone functions are in a bijective correspondence with connections, whereby a positive operator monotone function $f$ gives rise to a connection $\sigma=\sigma_f$ on $B(H)^{+}$ by the formula 
\begin{gather*}
    A\sigma B=A^{1/2}f(A^{-1/2}BA^{-1/2})A^{1/2},
\end{gather*}
for all $B\in B(H)^{+}, A\in B(H)^{++}$ which is extended uniquely to all $B(H)^{+}$ by property (III). Conversely, function $f$ can be recovered from a connection $\sigma$ by the formula
\begin{gather}
    f(x)I=I\sigma (xI),\; \text{for }x>0.\label{formula:RepresFuncToConn}
\end{gather}
The function $f$ is called the \textit{representing function} of $\sigma$. Furthermore, if $H$ is infinite dimensional then by a deep result of \cite{80KA}, each and every connection $\sigma$ on $B(H)^{+}$ arises in this way, i.e., $\sigma=\sigma_f$ for a unique operator monotone function $f$. Moreover, by \cite[Theorem 3.4]{80KA} (see also \cite[Lemma 3.1]{80KA}) it is known that, in terms of the representation \eqref{OperMonoFuncIntegralRepr},
\begin{gather*}
    A\sigma B =aA+bB+ \int_{(0,\infty)}\frac{1+\lambda}{\lambda}[(\lambda A):B]dm(\lambda),\;\text{for } A,B\in B(H)^+,
\end{gather*}
where $a=m(\{0\})$, $b=m(\{\infty\})$, $:$ denotes the parallel sum as def.\ by \eqref{def:ParallelSum}, and this establishes a bijection, $m\mapsto \sigma$, between the class of finite (positive) Borel measures on $[0,\infty]$ onto the class of connections.

For instance, by \eqref{formula:RepresFuncToConn} and \cite[Corollary 4.2]{80KA}, it is known that
a representing function $f$ corresponding to a connection $\sigma$ that is a mean or symmetric is one which satisfies, respectively,
\begin{gather*}
    f(1)=1;
    \quad f(x)=xf\left(\frac{1}{x}\right),\;\text{for }x>0.
\end{gather*}
For example, the representing functions $f_{\nabla}, f_!, f_{\#}$ of the arithmetic mean $\nabla,$ harmonic mean $!$, and geometric mean $\#$, respectively, are given by
\begin{gather}
\label{fnabla}
    f_{\nabla}(x)=\frac{1}{2}(1+x);\\
    \label{f!}
    f_!(x)=\frac{2x}{1+x};\\
    f_{\#}(x)=\sqrt{x}.
    \label{fgeom}
\end{gather}

A natural question arises that was considered by Kubo-Ando \cite[Sec.\ 4 and 5]{80KA} and, more recently, by L.\ Moln{\'a}r \cite{19LM}: What properties of a symmetric Kubo-Ando mean $\sigma=\sigma_f$ characterize it as the arithmetic mean $\nabla$, harmonic mean $!$, or geometric mean $\#$? Equivalently, in terms of representing functions $f$ (or, equivalently, on operator monotone functions $f$), what are necessary and sufficient conditions that guarantee $f\in\{f_{\nabla}, f_!, f_{\#}\}$? 

To address this question,  L.\ Moln{\'a}r \cite{19LM} considered an algebraic characterization of those means in terms of a weak form of an associativity law of a binary operation on $B(H)^{++}$. His main results in this regard can be summarized as follows (see \cite[Theorems 6 and 8]{19LM}).
\begin{Thm}[L.\ Moln{\'a}r]\label{thm:MolnarOpenProbRelThm}
    Let $H$ be a complex Hilbert space with $\dim H\geq 2$ and $\sigma $ be a symmetric Kubo-Ando mean on $B\left( H\right) ^{++}$ with representing operator monotone function $f$. Assume that there exists a continuous strictly
    increasing and surjective function $g:(0,\infty )\rightarrow (0,\infty )$
    such that the operation
    \begin{gather*}
        \diamond :\left( A,B\right) \mapsto g\left( A\sigma
    B\right) ,\;\text{for }A,B\in B\left( H\right) ^{++}
    \end{gather*}
    is either associative, i.e.,
    \begin{equation}
        \left( A\diamond C\right) \diamond B=A\diamond \left( C\diamond B\right) ,%
    \text{ }\forall A,B,C\in B\left( H\right) ^{++},\label{StrongFormAssoc}
\end{equation}%
or satisfies the weaker form of associativity
\begin{equation}
\left( A\diamond I\right) \diamond B=A\diamond \left( I\diamond B\right) ,%
\text{ }\forall A,B\in B\left( H\right) ^{++}.  \label{Thm6AssocId}
\end{equation}%
If (\ref{StrongFormAssoc}) is satisfied then $\sigma$ is the arithmetic or harmonic mean. On the other hand, if it satisfies (\ref{Thm6AssocId}) then either we have $g\left( f\left( x\right) \right) =x$, $x>0$ [meaning
that $A\diamond I=I\diamond A=A$, $A\in B\left( H\right) ^{++}$] or we have one of the following three possibilities:

\begin{itemize}
\item[(a)] there is a positive scalar $c\not=1$ such that $f\left(
c^{2}x\right) =cf\left( x\right) $, for $x>0$;

\item[(b)] $\sigma $ is the arithmetic mean;

\item[(c)] $\sigma $ is the harmonic mean.
\end{itemize}
\end{Thm}

He also proves (see \cite[pp.\ 160-161]{19LM}) the following partial converse of this theorem.
\begin{Thm}[L.\ Moln{\'a}r]
    If $\sigma $ is a symmetric Kubo-Ando mean with representing operator monotone function $f$ such that (a), (b), or (c) is true in Theorem \ref{thm:MolnarOpenProbRelThm} then there is a continuous strictly
    increasing and surjective function $g:(0,\infty )\rightarrow (0,\infty )$ with $g\not=f^{-1}$ such that the operation $\diamond :\left( A,B\right) \mapsto g\left( A\sigma
    B\right) $, $A,B\in B\left( H\right) ^{++}$ satisfies (\ref{Thm6AssocId}).
\end{Thm}

This motivates the following definition of a special class of positive operator means considered by L.\ Moln{\'a}r.
\begin{Def}[Moln{\'a}r mean]\label{def:MolnarMean}
    A symmetric Kubo-Ando mean $\sigma$ is called a Moln{\'a}r mean if its representing function $f$ has property (a) in Theorem \ref{thm:MolnarOpenProbRelThm}. The set of all Moln{\'a}r means will be called the Moln{\'a}r class of means.
\end{Def}

Notice that the geometric mean $\#$ is a Moln{\'a}r mean with representing function $f_{\#}$. Because of this, L.\ Moln{\'a}r \cite[p.\ 161]{19LM} posed the following problem, which we have rephrased in terms of Definition \ref{def:MolnarMean}.
\begin{Prob}[L.\ Moln{\'a}r]
    Is the geometric mean $\#$ the only Moln{\'a}r mean?
\end{Prob}

In \cite[p.~161]{19LM}, L.\ Moln{\'a}r says the following, in regard to the question above, which motivated our paper: ``However, we still do not know if the answer to the question is positive or negative. If it were affirmative, then we would get an interesting common characterization of the three fundamental operator means, the arithmetic, harmonic and geometric means." We are able to answer his question (in the negative) by proving the following theorem:
\begin{Thm}\label{Thm:SolnMolnarOpenProb}
    There are infinitely many Moln{\'a}r means.
\end{Thm}

The main goal of our paper is to prove this and, furthermore, to completely characterize the class of Moln{\'a}r means in terms of their representing functions, which we do with Theorems~\ref{th:Wpintrep} and \ref{th:Fourier}. These theorems are stated and proved in Section~\ref{sec:integr-repr-sw}; our approach is illustrated graphically in Fig.\ \ref{fig:summary}. Finally, we delve deeper into the order structure of the the Moln\'ar class (see Theorem~\ref{thm:order}) and give several infinite families of fully explicit nontrivial Moln{\'a}r means in Section \ref{sec:conclusion}. We conclude the paper by a modification of Moln{\'a}r’s original characterization so that the geometric mean is the only one satisfying the requisite set of properties (see Theorem \ref{th:CharGeoMean}).

\begin{figure}[t]
    \centering
    \includegraphics[scale=0.5]{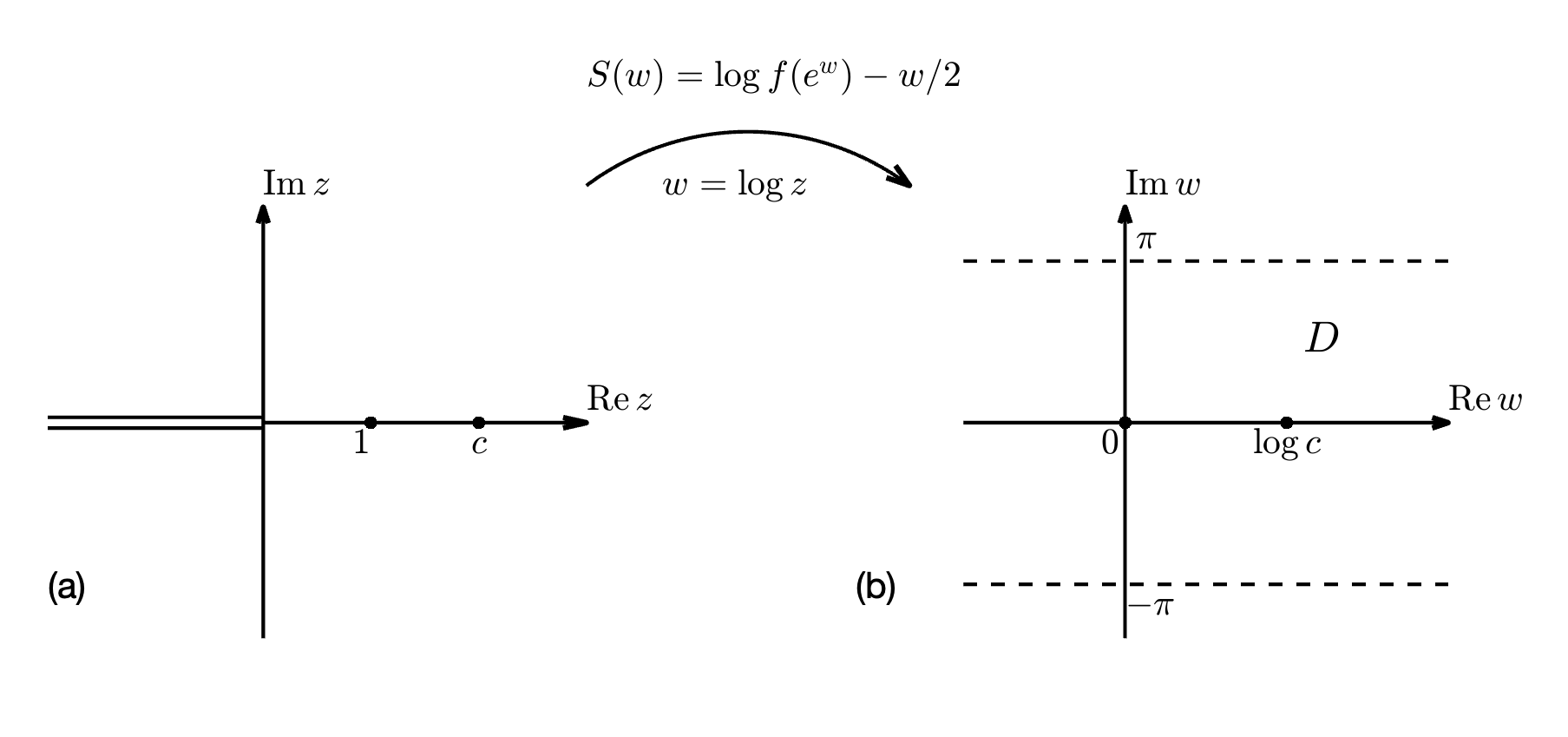}
    \caption{A Moln{\'a}r mean $\sigma=\sigma_f$ corresponds to a representing function $f\in \mathcal{M}_c$ for some $c\in (1,\infty)$. As represented in the transition from (a) to (b) in the figure, by using the invertible transform $w=\log z$ (with inverse $z=e^w$) from $\mathbb{C}\setminus (-\infty,0]$ onto the strip $D=\{w\in \mathbb{C}:|\operatorname{Im}w|<\pi\}=\mathbb{R}\times (-\pi,\pi)$, the analytic functions $f\in \mathcal{M}_c$ are mapped bijectively onto the analytic functions $S\in \mathcal{W}_p$ via $W_p(f)=\log f(e^{w})-w/2$, where $p=2\log c$. Then using the exponential representation of $OM_{+}\setminus\{0\}$ functions from \cite{77KN} we obtain the explicit representation of the class $\CW_{p}$ and, perforce, of $\CM_{c}$.}
    \label{fig:summary}
\end{figure}

\section{Characterization of the Moln{\'a}r class of means}\label{Sec2}
In this section we completely characterize the Moln{\'a}r class of means in terms of their representing functions. The next lemma is an important first step in this regard (whose proof is immediate from our discussion above) and motivates the definition that follows.  Also, as mentioned in the introduction, we will abuse notation and not distinguish between elements of $OM_+$ and positive operator monotone functions, unless necessary.

\begin{Lem}\label{lem:CharMolnarMeansRepresFunc}
    A connection $\sigma$ with representing function $f$ (i.e., $\sigma=\sigma_f$) is a Moln{\'a}r mean if and only if all of the following statements hold:
    \begin{itemize}
        \item[(i)] $f\in OM_+$;
        \item[(ii)] $f(1)=1$;
        \item[(iii)] $xf(1/x)=f(x),\text{ for }x>0$;
        \item[(iv)] there exists a positive scalar $c\not=1$ such that $f\left(c^{2}x\right) =cf\left( x\right) $, for $x>0$.
    \end{itemize}
\end{Lem}

\begin{Def}[Moln{\'a}r class of functions]\label{def:MolnarClassOfFunctions}
A function $f$ having properties (i)-(iv) in Lemma \ref{lem:CharMolnarMeansRepresFunc} will be called a Moln{\'a}r function. For a positive scalar $c\not=1$, we say $f$ is a Moln{\'a}r function of type $c$ if it is a Moln{\'a}r function and $f(c^2x)=cf(x),$ for $x>0$. The set of all Moln{\'a}r functions of type $c$ will be denoted by $\mathcal{M}_c$ and the Moln{\'a}r class of functions is $\mathcal{M}\equiv\cup_{c\in (0,\infty)\setminus\{1\}} \mathcal{M}_c$.
\end{Def}

As a consequence of the following lemma, to characterize the Moln{\'a}r means class $\mathcal{M}$, we need to consider only $c\in(1,\infty)$.
\begin{Lem}\label{lem:WhyWeCanAssumeCGeq1}
    For any positive scalar $c\not =1$,
    \begin{gather*}
        \mathcal{M}_c=\mathcal{M}_{1/c}.
    \end{gather*}
    In particular,
    \begin{gather*}
        \mathcal{M}=\cup_{c\in (1,\infty)} \mathcal{M}_c=\cup_{c\in (0,1)} \mathcal{M}_c.
    \end{gather*}
\end{Lem}
\begin{proof}
    Let $c$ be a positive scalar $c\not =1$. If $f(c^2x)=cf(x)$ for all $x>0$ then $x>0$ implies $c^{-2}x>0$ and hence $f(x)=f(c^2(c^{-2}x))=cf(c^{-2}x)$ so that $f((c^{-1})^2x)=c^{-1}f(x)$ for all $x>0$.
\end{proof}

Next, we denote the principal branches of the logarithm and the square root by $\log(\cdot)$ and $\sqrt{\cdot}$, respectively, i.e.,
\begin{gather*}
    \log(z)=\log|z|+i\operatorname{Arg}(z),\;
    \operatorname{Arg}(z)\in (-\pi,\pi), \; z\in \mathbb{C}\setminus(-\infty,0],  \\
    \sqrt{z}=e^{\frac{1}{2}\log(z)},\;z\in \mathbb{C}\setminus(-\infty,0].
\end{gather*}

Now have some preliminary results.
\begin{Lem}\label{Lem:OperMonotFuncArgControl}
    Let $f\in OM_+\setminus\{0\}$. Then
    \begin{gather*}
        \overline {f(\overline{z})}=f(z)\text{ and } f(z)\not=0,\text{ if }z\in \mathbb{C}\setminus(-\infty,0].
    \end{gather*}
    In addition,
    \begin{gather}
        0\leq \operatorname{Arg}f(z)\leq \operatorname{Arg}z<\pi, \text{ if }\operatorname{Im}z>0.\label{AngleConstraintOMPlusFunctions}
    \end{gather}
\end{Lem}
\begin{proof}
  Let $f\in OM_+\setminus\{0\}$. First, as $\overline {f(\overline{x})}=f(x)$ for all $x>0$, it follows by analyticity that $\overline {f(\overline{z})}=f(z)$ for all $z\in \mathbb{C}\setminus(-\infty,0]$. Next, it follows from the integral representation \eqref{OperMonoFuncIntegralRepr} of $f$ that $f(x)>0$ for all $x>0$. It also follows from the open mapping principle that, since $\Im(f(z))\ge 0$ for all $z\in\bb{C}^{+}$, then either $f(z)$ is a positive constant, or $\Im(f(z))>0$, for all $z\in\bb{C}^{+}$. Inequality (\ref{AngleConstraintOMPlusFunctions}) is proved by invoking a dual connection (cf.\ \cite[Corollary 4.3]{80KA} and \cite[p.\ 194]{10FH})
$f_{\perp}\in OM_+\setminus\{0\}$,  defined by:
    \begin{gather*}
        f_{\perp}(z)=\frac{z}{f(z)},\text{ for }z\in \mathbb{C}\setminus(-\infty,0].
    \end{gather*}
    As $f,f_{\perp}\in OM_+\setminus\{0\}$, then the compositions $\log\circ f$ and $\log\circ f_{\perp}$ are analytic functions on $\mathbb{C}\setminus(-\infty,0]$ satisfying $(\log\circ f_{\perp})(x)=\log(x)-\log f(x)$ for all $x>0$. By analyticity this implies $\log\circ f_{\perp}=\log-\log\circ f$ on $\mathbb{C}\setminus(-\infty,0]$ so that
    \begin{gather*}
        0\leq \operatorname{Im}[\log f_{\perp}(z)]=\operatorname{Im}(\log z)-\operatorname{Im}[\log f(z)]=\operatorname{Arg}z-\operatorname{Arg}f(z)
    \end{gather*}
    for every $z\in \mathbb{C}^{+}$, proving inequalities \eqref{AngleConstraintOMPlusFunctions}.
\end{proof}

\begin{Lem}\label{lem:DefAndPropOfSFunc}
    If $f\in OM_+\setminus\{0\}$ then the function $S:D\rightarrow \mathbb{C}$ defined by
\begin{gather}
    S(w)=\log f(e^w)-\frac{1}{2}w, \text{ for }w\in D,\label{DefSfInTermsOffPart1}\\
    D=\{w\in \mathbb{C}:-\pi <\operatorname{Im}w<\pi\},\label{DefSfInTermsOffPart2}
\end{gather}
is analytic on $D$ and satisfies
\begin{gather}
    |\operatorname{Im}S(w)|\leq\frac{1}{2}\operatorname{Im}w,\;\text{if }0\leq \operatorname{Im} w<\pi,\label{ImagPartSConstr}\\
     f(z)=\sqrt{z}e^{S(\log z)},\;\text{for } z\in \mathbb{C}\setminus(-\infty,0].\label{RecovFormulaForfFromS}
\end{gather}
Conversely, if $S:D\rightarrow \mathbb{C}$ is an analytic function satisfying (\ref{ImagPartSConstr}) then the function $f:\mathbb{C}\setminus(-\infty,0]\rightarrow \mathbb{C}$ defined by (\ref{RecovFormulaForfFromS}) is in $OM_+\setminus\{0\}$ and $S$ is given in terms of $f$ by (\ref{DefSfInTermsOffPart1}).
\end{Lem}
\begin{proof}
    ($\Rightarrow$): Let $f\in OM_+\setminus\{0\}$. Then by Lemma \ref{Lem:OperMonotFuncArgControl}, $f:\mathbb{C}\setminus(-\infty,0]\rightarrow \mathbb{C}$ is an analytic function with range $f(\mathbb{C}\setminus(-\infty,0])\subseteq \mathbb{C}\setminus(-\infty,0]$ so that the function $z\mapsto \log f(z)$ is analytic on the domain $\mathbb{C}\setminus(-\infty,0]$.  Next, the function $w\mapsto e^{w}$ is analytic on the domain $D$ with range $e^{D}=\mathbb{C}\setminus(-\infty,0]$. It follows that the composition of functions $w\mapsto \log f(e^w)$ is well-defined and analytic on $D$. From this it follows immediately that the function $S:D\rightarrow \mathbb{C}$ defined by (\ref{DefSfInTermsOffPart1}), (\ref{DefSfInTermsOffPart2}) is analytic on $D$ and satisfies (\ref{RecovFormulaForfFromS}). Also, if $\operatorname{Im}w=0$ then $w\in D$ and since $f(x)\geq 0$ for $x>0$, then by (\ref{DefSfInTermsOffPart1}) we have $\operatorname{Im}S(w)=\operatorname{Arg}f(e^{\operatorname{Re}w})=0$. To complete the proof of (\ref{ImagPartSConstr}), we note that by the hypotheses on $f$ it follows from Lemma \ref{Lem:OperMonotFuncArgControl} that $0\leq \operatorname{Arg}f(z)\leq \operatorname{Arg}z<\pi$ if $\operatorname{Im}z>0$ and since
    \begin{gather*}
        \operatorname{Arg}e^w=\operatorname{Im}w,\;\operatorname{Im}e^w=e^{\operatorname{Re}w}\sin(\operatorname{Im}w)>0,\;\text{if }0< \operatorname{Im} w<\pi,
    \end{gather*}
    it follows that
    \begin{gather*}
    -\frac{1}{2}\operatorname{Im}w\leq \operatorname{Im}S(w)=\operatorname{Arg}f(e^w)-\frac{1}{2}\operatorname{Im}w\leq \frac{1}{2}\operatorname{Im}w,\;\text{if }0< \operatorname{Im} w<\pi,
    \end{gather*}
    which proves (\ref{ImagPartSConstr}).

    ($\Leftarrow$): Conversely, suppose $S:D\rightarrow \mathbb{C}$ is an analytic function [where $D$ is defined by (\ref{DefSfInTermsOffPart2})] satisfying (\ref{ImagPartSConstr}). Let $f:\mathbb{C}\setminus(-\infty,0]\rightarrow \mathbb{C}$ be the function defined by (\ref{RecovFormulaForfFromS}). As $\log z$ and $\sqrt{z}$ are analytic on the domain $\mathbb{C}\setminus(-\infty,0]$ with $\log z\in D$ for $z\in \mathbb{C}\setminus(-\infty,0]$ then it follows that $f:\mathbb{C}\setminus(-\infty,0]\rightarrow \mathbb{C}$ is analytic and cannot be the zero function since $f(z)=\sqrt{z}e^{S(\log z)}\not=0$ for all $z$ in the domain of $f$. Next, (\ref{ImagPartSConstr}) implies that $S(\log x)\in \mathbb{R}$ for $x>0$ so that $f(x)=\sqrt{x}e^{S(\log x)}>0$ for $x>0$. Also, (\ref{ImagPartSConstr}) implies that for $\operatorname{Im} z>0$ we have
    \begin{gather}
       0\leq \operatorname{Im}\left[S(\log z)+\frac{1}{2}\log z\right]\leq \operatorname{Im}\log z=\operatorname{Arg}z<\pi\label{KeyIneqImPartSFuncComposLog}
    \end{gather}
    and hence
    \begin{gather*}
        f(z)=\sqrt{z}e^{S(\log z)}=e^{S(\log z)+\frac{1}{2}\log z},\\
        \operatorname{Im}f(z)=e^{\operatorname{Re}[S(\log z)+\frac{1}{2}\log z]}\sin\left\{\operatorname{Im}\left[S(\log z)+\frac{1}{2}\log z\right]\right\}\geq 0.
    \end{gather*}
    Now to complete the proof, it remains only to prove that the function $S:D\rightarrow \mathbb{C}$ is given in terms of $f$ by the formula (\ref{DefSfInTermsOffPart1}). It follows by (\ref{RecovFormulaForfFromS}) that there exists a constant $m\in\mathbb{Z}$ such that
    \begin{gather}
        \log f(z)=S(\log z)+\frac{1}{2}\log z+i2\pi m, \text{ for } z\in \mathbb{C}\setminus(-\infty,0].\label{KeyEqSFuncComposLog}
    \end{gather}
    This implies that
    \begin{gather*}
        2\pi m+\operatorname{Im}\left[S(\log z)+\frac{1}{2}\log z\right]=\operatorname{Im} \log f(z)=\operatorname{Arg}f(z), \text{ for } z\in \mathbb{C}\setminus(-\infty,0].
    \end{gather*}
    As we know that if $\operatorname{Im} z>0$ then (\ref{KeyIneqImPartSFuncComposLog}) holds and $\operatorname{Arg} f(z)\in [0,\pi)$, this implies $m=0$. Hence, from (\ref{KeyEqSFuncComposLog}) it follows that
    \begin{gather*}
        \log f(e^w)=S(\log e^w)+\frac{1}{2}\log e^w=S(w)+\frac{1}{2}w, \text{ for } w\in D,
    \end{gather*}
    which proves equality (\ref{DefSfInTermsOffPart1}). This completes the proof.
\end{proof}

\begin{Lem}\label{lem:FundPropertiesRelatingMolnarFunctionsToSFunctions}
    Let $f\in  OM_+\setminus\{0\}$ and denote by $S:D\rightarrow \mathbb{C}$ the corresponding function defined by (\ref{DefSfInTermsOffPart1}). Then the following statements are true:
    \begin{itemize}
        \item[(i)] $f(1)=1$ if and only if $S(0)=0$.
        \item[(ii)] $f(x)=xf(1/x),$ for $x>0$ if and only if $S(-w)=S(w), \text{ for } w\in D$.
        \item[(iii)] If $c\in (1,\infty)$ then $f\left(c^{2}x\right) =cf\left( x\right) $, for $x>0$ if and only if $S(w+2\log c)=S(w),$ for $w\in D$ \textnormal{[}i.e., $S$ is periodic with a period $2\log c$\textnormal{]}, in which case
        \begin{gather*}
            \lim_{x\rightarrow 0^+}f(x)=0.
        \end{gather*}
    \end{itemize}
\end{Lem}
\begin{proof}
    (i): If $f(1)=1$ then $S(0)=\log f(e^0)-\frac{1}{2}(0)=\log 1=0.$ Conversely, if $S(0)=0$ then $f(1)=\sqrt{1}e^{S(\log 1)}=e^{S(0)}=1$.

    (ii): Suppose $f(x)=xf(1/x)$ for all $x>0$. If $w\in \mathbb{R}$ then
\begin{gather*}
    S(-w)=\log f(e^{-w})-\frac{1}{2}(-w)=\log [e^{-w}f(e^w)]+\frac{1}{2}w\\
    =\log [f(e^w)]+\log[e^{-w}]+\frac{1}{2}w=\log [f(e^w)]-\frac{1}{2}w
    =S(w).
  \end{gather*}
 It follows that $S(w)=S(-w)$ for all $w\in D$, since functions $S(w)$ and $S(-w)$ are analytic in $D$ and agree on $\bb{R}\subset D$. Conversely, suppose that  $S(w)=S(-w)$ for all $w\in D$. Then, for every $x>0$,
\begin{gather*}
    xf(1/x)=x\sqrt{x^{-1}}e^{S(\log (x^{-1}))}=\sqrt{x}e^{S(-\log x)}=\sqrt{x}e^{S(\log x)}=f(x).
\end{gather*}

(iii): Let $c\in (1,\infty)$. Suppose $f\left(c^{2}x\right) =cf\left( x\right) $ for all $x>0$. If $w\in \mathbb{R}$ then
\begin{gather*}
    S(w+2\log c)=\log f(e^{w+2\log c})-\frac{1}{2}(w+2\log c)\\
    =\log f(c^2e^{w})-\log c-\frac{1}{2}w=\log [cf(e^{w})]-\log c-\frac{1}{2}w\\
    =\log c+\log f(e^{w})-\log c-\frac{1}{2}w
    =S(w).
  \end{gather*}
 It follows that $S(w+2\log c)=S(w)$ for all $w\in D$,  since functions $S(w)$ and $S(w+2\log c)$ are analytic in $D$ and agree on $\bb{R}\subset D$.
Conversely, suppose $S(w+2\log c)=S(w)$ for all $w\in D$. Then, for every $x>0$,
\begin{gather*}
    f(c^2x)=\sqrt{c^2x}e^{S(\log(c^2x))}=c\sqrt{x}e^{S(\log x+2\log c)}=c\sqrt{x}e^{S(\log x)}=cf(x).
\end{gather*}
Thus, in particular, if this is the case, then the restriction $S|_{\mathbb{R}}:\mathbb{R}\rightarrow \mathbb{R}$ is a continuous periodic function on $\mathbb{R}$, implying it is bounded, so that $x\mapsto e^{S(\log x)}$ is a bounded function on $(0,\infty)$ and therefore,
\begin{gather*}
            \lim_{x\rightarrow 0^+}f(x)=\lim_{x\rightarrow 0^+}\sqrt{x}e^{S(\log x)}=0.
\end{gather*}
(For alternative proof of the fact $\lim_{x\rightarrow 0^+}f(x)=0$, cf.\ \cite[pp.\ 160-161]{19LM}.) This completes the proof.
\end{proof}

These results motivate the following definition and proposition.
\begin{Def}[$\mathcal{W}$-functions]\label{def:WClassOfFunctions}
Let $D=\{w\in \mathbb{C}:-\pi <\operatorname{Im}w<\pi\}$. An analytic function $S:D\rightarrow \mathbb{C}$ having the properties:
\begin{itemize}
    \item[(i)] $|\operatorname{Im}S(w)|\leq\frac{1}{2}\operatorname{Im}w,\;\text{if }0\leq \operatorname{Im} w<\pi,$
    \item[(ii)] $S(0)=0$, 
    \item[(iii)] $S(-w)=S(w)$ for all $w\in D$,
    \item[(iv)] there exists a scalar $p>0$ such that $S(w+p)=S(w)$ for all $w\in D$,
\end{itemize}
will be called a $\mathcal{W}$-function with period $p$. The set of all $\mathcal{W}$-functions with period $p$ will be denoted by $\mathcal{W}_p$ and the class of all $\mathcal{W}$-functions is $\mathcal{W}=\cup_{p\in (0,\infty)} \mathcal{W}_p$.
\end{Def}

\begin{Pro}
    For each $f\in \mathcal{M}$, denote by $S_f$ the function $S$ defined by \eqref{DefSfInTermsOffPart1} and \eqref{DefSfInTermsOffPart2}.  Then the map
    \begin{gather*}
        W:\mathcal{M}\rightarrow \mathcal{W},\\
        W(f)=S_f,\; f\in \mathcal{M}
    \end{gather*}
    is a bijection from $\mathcal{M}$ onto $\mathcal{W}$ such that
    \begin{gather*}
        f\in \mathcal{M}_c\Leftrightarrow S_f\in \mathcal{W}_p,
    \end{gather*}
    where $c\in (1,\infty), p\in (0,\infty)$ are related by the formula:
    \begin{gather*}
        p=2\log c.
    \end{gather*}
    In particular, the restriction map $W_p$ defined by
    \begin{gather*}
        W_p=W|_{\mathcal{M}_{c}}:\mathcal{M}_c\rightarrow \mathcal{W}_p,\\
        W_p(f)=W(f)=S_f,\; f\in \mathcal{M}_c
    \end{gather*}
    is a bijection from $\mathcal{M}_c$ onto $\mathcal{W}_p$.
\end{Pro}
\begin{proof}
    The proof of this proposition follows immediately from Definitions \ref{def:MolnarClassOfFunctions} and \ref{def:WClassOfFunctions} and Lemmas \ref{lem:CharMolnarMeansRepresFunc}, \ref{lem:DefAndPropOfSFunc}, and \ref{lem:FundPropertiesRelatingMolnarFunctionsToSFunctions}.
\end{proof}

\section{Explicit characterization of $\CW_{p}$}
\label{sec:integr-repr-sw}
The goal of this section is to characterize the class of all $\mathcal{W}_{p}$-functions explicitly. Before we do this, we will need the following well-known collection of results (see, for instance, \cite{56AD, 64AD, 00GT, 05DK, 14GT}) on the relationship between Herglotz functions and their boundary values.
\begin{Thm}[Herglotz integral representation]\label{th:HerglotzIntReprWithImplicitSokhotskiPlemeljFormulas}
    If $a\in \mathbb{R}, b\geq 0,$ and $\mu$ is a positive Borel measure on $\mathbb{R}$ such that $\int_{\mathbb{R}}(1+\lambda^2)^{-1}\mu(\lambda)<\infty$ then
    \begin{gather*}
        h(z)=a+bz+\int_{\mathbb{R}}\left(\frac{1}{\lambda-z}-\frac{\lambda}{1+\lambda^2}\right)d\mu(\lambda),\quad z\in \mathbb{C}^+.
    \end{gather*}
    is a Herglotz function. The map $(a,b,\mu)\mapsto h$ defines a bijection between the class of all such triples and the set of all Herglotz functions. In particular, the triple $(a,b,\mu)$ can be recovered from the Herglotz function $h$ by the formulas
    \begin{gather*}
        a=\operatorname{Re}h(i),\quad b=\lim_{\eta\uparrow \infty}\frac{h(i\eta)}{i\eta},\\
        \frac{1}{2}\mu(\{\lambda_1\})+\frac{1}{2}\mu(\{\lambda_2\})+\mu((\lambda_1,\lambda_2))=\frac{1}{\pi}\lim_{\varepsilon\downarrow 0}\int_{\lambda_1}^{\lambda_2}\operatorname{Im}h(\lambda+i\varepsilon)d\lambda,\quad (\lambda_1,\lambda_2)\subseteq \mathbb{R}.
    \end{gather*}
    Moreover, $h(z)$ has finite normal limits $\lim_{\varepsilon\downarrow 0}h(\lambda+i\varepsilon)$ for Lebesgue a.e.\ $\lambda\in \mathbb{R}$ and the absolutely continuous (ac) part $\mu_{ac}$ of the measure $\mu$ with respect to the Lebesgue measure on $\mathbb{R}$ has density function (i.e., Radon-Nikodym derivative)
    \begin{gather*}
        \frac{d\mu_{ac}(\lambda)}{d\lambda}=\frac{1}{\pi}\lim_{\varepsilon\downarrow 0}\operatorname{Im}h(\lambda+i\varepsilon),\;\;\text{for Lebesgue a.e.\ }\lambda\in \mathbb{R}.
    \end{gather*}
    Furthermore, this latter statement is also true if we replace the normal limit to $\lambda$ by the limit to $\lambda$ in any (non-tangential) circular sector in the upper-half plane.
\end{Thm}

Our next result is an integral representation that shows that the class $\CW_{p}$ can be parametrized by all measurable, real-valued, odd, periodic functions bounded by 1/2 in absolute value.
\begin{Thm}
  \label{th:Wpintrep}
  If $\Psi\in L^{\infty}(\bb{R})$ is a real-valued, odd, $p$-periodic function with $\|\Psi\|_{\infty}\le 1/2$ then
  \begin{equation}
  \label{SGint}
    S(w)=\frac{1}{2}\int_{\bb{R}}\frac{\Psi(\Gl)\sinh\left(\frac{w}{2}\right) d\Gl}{\cosh\left(\frac{\Gl}{2}\right)\cosh\left(\frac{w-\Gl}{2}\right)},\;\;w\in D 
  \end{equation}
  is a function belonging to $\CW_p$. The map $\Psi \mapsto S$ defines a bijection between the class of all such functions in $L^{\infty}(\bb{R})$ (with equality in the sense of the $||\cdot||_{\infty}$ norm) and the set $\CW_p$. In particular, if $S\in \CW_p$ then $\Psi$ can be uniquely recovered by the formula:
  \begin{equation}
    \label{S2G}
    \Psi(\Gl)=\nth{\pi}\lim_{\mu\to\pi^{-}}\operatorname{Im}S(\Gl+i\mu),\;\;\text{for Lebesgue a.e.\ }\lambda\in \mathbb{R}.
  \end{equation}
\end{Thm}
\begin{proof}
We start with the exponential representation of a function $f\in OM_{+}\setminus\{0\}$ from \cite[Appendix, Sec.\ 5]{77KN} (see also \cite{56AD, 00GT}): There exists scalar $C$ and function $\phi\in L^{\infty}((0,\infty))$ such that for every $z\in \mathbb{C}\setminus (-\infty,0]$,
\begin{equation}
  \label{mult}
      f(z)=Ce^{g(z)},\quad C>0, \quad
    g(z)=\int_{0}^{\infty}\left(\frac{t}{1+t^{2}}-\frac{1}{t+z}\right)\phi(t)dt,\quad 0\le \phi(t)\le 1. 
  \end{equation}
We observe that if $z\not\in\bb{R}$ then
  \[
    \left|\operatorname{Im}g(z)\right|=
    |\operatorname{Im} z|\int_{0}^{\infty}\frac{\phi(t)dt}{|t+z|^{2}}<\int_{\bb{R}}\frac{|\operatorname{Im} z|dt}{|t+z|^{2}}=\pi.
  \]
  Thus, for all $z\in\mathbb{C}\setminus(-\infty,0]$,
\[
  S(\log z)=\log f(z)-\frac{1}{2}\log z=\log C -\frac{1}{2}\log z+
  g(z).
\]
Observing that
\[
  \int_{0}^{\infty}\left(\frac{t}{1+t^{2}}-\nth{t+z}\right)dt=\log z,
\]
we obtain 
\begin{equation}
    \label{altSRepr}
S(\log z)=\log C+\int_{0}^{\infty}\left(\frac{t}{1+t^{2}}-\nth{t+z}\right)\psi(t)dt,\quad ||\psi||_{\infty}\le\hf,
\end{equation}
where $\psi(t)=\phi(t)-1/2$.
For $S$ to be in $\CW_{p}$ we also need $S(0)=0$, $S(-w)=S(w)$ and $S(w+p)=S(w)$.
Clearly, $S(0)=0$ is equivalent to
\[
\log C=-\int_{0}^{\infty}\left(\frac{t}{1+t^{2}}-\nth{t+1}\right)\psi(t)dt.
\]
Hence,
\begin{equation*}
  S(\log z)=\int_{0}^{\infty}\left(\frac{1}{1+t}-\nth{t+z}\right)\psi(t)dt
=\int_{0}^{\infty}\frac{(z-1) \psi(t)}{(t+1)(t+z)}dt,
\end{equation*}
and $S(-w)=S(w)$ is equivalent to
\begin{equation}
  \label{odd}
  \int_{0}^{\infty}\frac{(z-1) \psi(t)}{(t+1)(t+z)}dt=\int_{0}^{\infty}\frac{(1-z) \psi(t)}{(t+1)(zt+1)}dt.
\end{equation}
Making a change of variables $s=1/t$ in the integral on the \rhs\ in (\ref{odd}), and denoting $\Tld{\psi}(s)=\psi(1/s)$, equation (\ref{odd}) becomes
\[
\int_{0}^{\infty}\frac{\psi(t)}{(t+1)(t+z)}dt=-\int_{0}^{\infty}\frac{\Tld{\psi}(s)}{(s+1)(z+s)}ds,
\]
which implies $\psi(t)=-\Tld{\psi}(t)=-\psi(1/t)$ for a.e.\ $t\in (0,\infty)$ (by Theorem \ref{th:HerglotzIntReprWithImplicitSokhotskiPlemeljFormulas}).
Hence, writing
\[
S(\log z)=\int_{0}^{1}\frac{(z-1) \psi(t)}{(t+1)(t+z)}dt+\int_{1}^{\infty}\frac{(z-1) \psi(t)}{(t+1)(t+z)}dt,
\]
and changing the variable of integration $s=1/t$ in the second term, we obtain
\[
    S(\log z)=\int_{0}^{1}\frac{(z-1) \psi(t)}{(t+1)(t+z)}dt-\int_{0}^{1}\frac{(z-1) \psi(s)}{(s+1)(1+zs)}ds.
\]
Equivalently, we can write, using partial fraction decomposition,
\begin{equation}
  \label{S}
S(\log z)=\int_{0}^{1}\left(\frac{2}{t+1}-\frac{1}{t+z}-\frac{z}{1+tz}\right)\psi(t)dt.
\end{equation}
Finally, $S(w+p)=S(w)$ is equivalent to
\[
\int_{0}^{1}\left(\frac{1}{t+z}+\frac{z}{1+tz}\right)\psi(t)dt=\int_{0}^{1}\left(\frac{1}{t+c^{2}z}+\frac{c^{2}z}{1+tc^{2}z}\right)\psi(t)dt,
\]
where $p=2\log c$.
Putting everything on one side, we obtain
\begin{equation}
  \label{eqn} \int_{0}^{1}\left(\frac{1}{t+c^{2}z}+\frac{c^{2}z}{1+tc^{2}z}-\frac{1}{t+z}-\frac{z}{1+tz}\right)\psi(t)dt=0,
  \quad\forall z\in\bb{C}\setminus(-\infty,0].
\end{equation}
To understand (\ref{eqn}) we want to rewrite it, so that each term
above is a Cauchy-type integral. We accomplish this by
changing variables $t=c^{2}s$ in the first term, and $t=c^{-2}s$, in the second.
Then, upon switching $z$ to $-z$, we obtain the following identity on $\bb{C}\setminus[0,\infty)$:
\begin{equation*}
  F(z):=F_1(z)+F_2(z)=0,
\end{equation*}
where 
\begin{equation*}
    F_1(z)=\int_{0}^{1}\frac{\psi(c^{2}s)\chi_{[0,c^{-2}]}(s)-\psi(s)}{s-z}ds,\;F_2(z)=\int_{0}^{c^{2}}\frac{\psi(sc^{-2})-\psi(s)\chi_{[0,1]}(s)}{s-z^{-1}}ds.
\end{equation*}
It now follows from this (and by Theorem \ref{th:HerglotzIntReprWithImplicitSokhotskiPlemeljFormulas})
that for a.e.\ $x\in(0,\infty)$:
\begin{equation}
    \label{per}0=\frac{1}{\pi}\lim_{y\downarrow 0}\operatorname{Im}F(x+iy)=\begin{cases}
  \psi(xc^{2})-\psi(x),&x\in(0,c^{-2}),\\
  -\psi(x)-\psi((c^{2}x)^{-1}), &x\in(c^{-2},1),\\
  \psi(x^{-1})-\psi((c^{2}x)^{-1}), &x\in(1,\infty).\\
  \end{cases}
\end{equation}
Now define $\Psi(\Gl)=\psi(e^{\Gl})$ for $\Gl\in\bb{R}$.
Then, writing $p=2\log c$, we obtain from \eqref{per} that the following equations hold for a.e.\ $\lambda\in \mathbb{R}$:
\begin{equation}
  \label{G}
    \begin{cases}
    \Psi(\Gl+p)=\Psi(\Gl),&\Gl<-p,\\
    \Psi(\Gl)=-\Psi(-\Gl-p),&-p<\Gl<0,\\
    \Psi(-\Gl)=\Psi(-\Gl-p),&\Gl>0.
  \end{cases}
\end{equation}
Recalling that $\psi(x)=-\psi(1/x)$, for a.e. $x>1$, we obtain $\Psi(\Gl)=-\Psi(-\Gl)$ for a.e.\ $\Gl>0$.
It follows that $\Psi\in L^{\infty}(\bb{R})$ is a real-valued, odd, $p$-periodic function with $\|\Psi\|_{\infty}\le 1/2$. 
Finally, by substituting in $z=e^w, w\in D$ into \eqref{odd} and changing variables $t=e^{\lambda}$, we arrive at the representation \eqref{SGint} for $S(w)$.

We now claim that any such function $\Psi(\Gl)$ implies that $F(z)=0$ for all $z\in\bb{C}\setminus[0,\infty)$. Indeed, changing variables under the integral $s=e^{\Gl}$, we obtain
\[
  F(z)=\int_{-\infty}^{0}\frac{\Psi(\Gl+p)\chi_{(-\infty,-p)}(\Gl)-\Psi(\Gl)}{1-e^{-\Gl}z}d\Gl
  +\int_{-\infty}^{p}\frac{\Psi(\Gl-p)-\Psi(\Gl)\chi_{(-\infty,0)}(\Gl)}{1-e^{-\Gl}z^{-1}}d\Gl.
\]
Using the properties of the function $\Psi(\Gl)$ we obtain
\[
  F(z)=-\int_{-p}^{0}\frac{\Psi(\Gl)}{1-e^{-\Gl}z}d\Gl+\int_{0}^{p}\frac{\Psi(\Gl)}{1-e^{-\Gl}z^{-1}}d\Gl
\]
Changing variables $\Gl=-\Gl'$ in the first integral and using the fact that $\Psi(\Gl)$ is odd, we obtain
\[
F(z)=  \int_{0}^{p}\Psi(\Gl)\left(\nth{1-e^{\Gl}z}+\nth{1-(e^{\Gl}z)^{-1}}\right)d\Gl=\int_{0}^{p}\Psi(\Gl)d\Gl.
\]
Since $\Psi(\Gl)$ is odd and $p$-periodic we have
\[
\int_{0}^{p}\Psi(\Gl)d\Gl=\int_{-p/2}^{p/2}\Psi(\Gl)d\Gl=0.
\]
This proves the claim. We can also conclude from this that for any $\Psi\in L^{\infty}(\mathbb{R})$ which is a real-valued, odd, $p=2\log c$-periodic function such that $||\Psi|||_{\infty}\le 1/2$, the function
\begin{equation}
  \label{f}
  f(z)=\sqrt{z}\exp\left\{(z-1)\int_{\bb{R}}\frac{\Psi(\Gl)}{(1+e^{\Gl})(1+e^{-\Gl}z)}d\lambda\right\},\;\;z\in \mathbb{C}\setminus (-\infty,0]
\end{equation}
is in the class $\CM_{c}$ and the function $S(w)$ given by (\ref{SGint}) belongs to $\CW_p$.

Finally, given $S(w)$ in $\CW_{p}$ we can compute the corresponding
spectral function $\Psi(\Gl)=\psi(e^{\Gl})$ using the representation (\ref{altSRepr}):
\[
-S(\log(-z))=-\log C+\int_{0}^{\infty}\left(\nth{t-z}-\frac{t}{1+t^{2}}\right)\psi(t)dt,\;\;z\in \mathbb{C}\setminus[0,\infty).
\]
Then it follows from this (by Theorem \ref{th:HerglotzIntReprWithImplicitSokhotskiPlemeljFormulas}) that
\begin{equation*}
     \psi(x)=\frac{-1}{\pi}\lim_{\eta\rightarrow 0^+}\operatorname{Im}S(\log(-xe^{i\eta}))=\frac{-1}{\pi}\lim_{\eta\rightarrow 0^+}\operatorname{Im}S(\log x+i(\eta-\pi))=\frac{1}{\pi}\lim_{\eta\rightarrow 0^+}\operatorname{Im}S(\log x+i(\pi-\eta))
\end{equation*}
for a.e.\ $x\in (0,\infty)$, where we also used the identity $\bra{S(w)}=S(\bra{w})$ for all $w\in D$. 
Hence,
\begin{equation}
  \label{g}
    \psi(x)=\frac{1}{\pi}\lim_{\mu\to\pi^{-}}\operatorname{Im} S(i\mu+\log x)
  \end{equation}
  for a.e.\ $x\in (0,\infty)$, from which formula \eqref{S2G} now follows.
\end{proof}
Theorem~\ref{th:Wpintrep} shows that all classes $\CM_{c}$, $c>1$ have infinitely many members. However, the integral representation (\ref{SGint}) does not permit us to exhibit functions in $\CM_{c}$, or equivalently, in $\CW_{p}$, explicitly. This is remedied by our next result.
\begin{Thm}
  \label{th:Fourier}
  Suppose
  \begin{equation}
    \label{PsiFourier}
    \Psi(\Gl)=\sum_{n=1}^{\infty}B_{n}\sin(an\Gl) 
  \end{equation}
is a Fourier series of a real, odd, $p=2\pi/a$-periodic function $\Psi(\Gl)$, satisfying $|\Psi(\Gl)|\le1/2$. Then, the corresponding function $S\in\CW_{p}$ in (\ref{SGint}) is given by the Fourier series
\begin{equation}
  \label{SGF}
  S(w)=\pi\sum_{n=1}^{\infty}B_{n}\frac{1-\cos(awn)}{\sinh(a\pi n)}=2\pi \sum_{n=1}^{\infty}B_{n}\frac{\sin^{2}\left(\frac{awn}{2}\right)}{\sinh(a\pi n)}.
\end{equation}
\end{Thm}
\begin{proof}
  We begin with the observation that the functions $S_{n}(w)=A_{n}(1-\cos(awn))$, $n\in\bb{N}$, are entire, even, $p=2\pi/a$-periodic, and satisfy $S_{n}(0)=0$. Thus, $S_{n}\in\CW_{p}$, if we can show that $S_{n}(w)$ has property (i) in Definition~\ref{def:WClassOfFunctions}. We compute
  \[
\Im S_{n}(w)=A_{n}\sin(a\Gl n)\sinh(a\mu n),\quad w=\Gl+i\mu.
\]
Since $\sinh(a\mu n)$ is convex on $\mu\in[0,\pi]$ we have
\[
\sinh(a\mu n)\le\frac{\sinh(a\pi n)}{\pi}\mu.
\]
Since, $\sin(a\mu n)$ is an odd function of $\mu$, we conclude that
\[
|\Im S_{n}(w)|\le |A_{n}|\frac{\sinh(a\pi n)}{\pi}|\Im\,w|.
\]
Thus, choosing $A_{n}=\pi/(2 \sinh(a\pi n))$, we obtain $S_{n}\in\CW_{p}$. Formula (\ref{S2G}) gives
\[
\Psi_{n}(\Gl)=\frac{1}{2 \sinh(a\pi n)}\lim_{\mu\to\pi^{-}}\sin(a\Gl n)\sinh(a\mu n)=\hf\sin(a\Gl n).
\]
By linearity of the representation (\ref{SGint}) we conclude that if $\Psi$ is given by (\ref{PsiFourier}), then the corresponding $S(w)$ must be given by (\ref{SGF}). By Theorem~\ref{th:Wpintrep}, if $|\Psi(\Gl)|\le 1/2$, then $S\in\CW_{p}$.
\end{proof}

\section{Conclusions}\label{sec:conclusion}
From our results, we are able to make several important conclusions.
One is a corollary of Theorem~\ref{th:Fourier} that provides a sequence of explicit non-geometric Moln{\'a}r means, corresponding to $\Psi(\Gl)=\pm(1/2)\sin(\pi\Gl/\log c)$:
\begin{equation}
f_{n}(x)=\sqrt{x}\exp\left\{\frac{\pi\sin^{2}\left(\frac{\pi n\log x}{2\log c}\right)}{\sinh\left(\frac{\pi^{2}n}{\log c}\right)}\right\}\in\CM_{c},\quad n\in\bb{Z}\setminus\{0\},\label{SeqNonGeomMolnarMeans}
\end{equation}
or recalling that the Moln{\'a}r class $\CM$ is the union of all $\CM_c$, the family
\begin{equation}
f_{\Ga}(x)=\sqrt{x}\exp\left\{\frac{\pi\sin^{2}(\Ga\log x)}{\sinh(2\pi\Ga)}\right\}\in\CM,\quad\Ga\in\bb{R}\setminus\{0\}.\label{FamNonGeomMolnarMeans}
\end{equation}

Second, we observe that any real, odd, $p$-periodic function $\Psi(\Gl)$, satisfying $|\Psi(\Gl)|\le1/2$ is uniquely determined by its restriction to the half-period interval $(0,p/2)$. Using the symmetries of $\Psi(\Gl)$, we can rewrite the exponential-integral representation (\ref{f}) as 
\begin{equation}
  \label{fEllip}
  f(z)=\sqrt{z}\exp\left\{\int_{0}^{p/2}\Psi(\Gl)E_{p}(\Gl;z)d\Gl\right\},\quad\Psi\in B(0,1/2)\subseteq 
  L^{\infty}(0,p/2),
\end{equation}
where for each fixed $z\in\bb{C}\setminus(-\infty,0]$,
\begin{equation}
  E_{p}(\Gl;z)=\sum_{n\in\bb{Z}}\frac{(z-1)^{2}(e^{\Gl+pn}-1)}{(e^{\Gl+pn}+1)(z+e^{\Gl+pn})(z+e^{-(\Gl+pn)})},
  \label{Ep}
\end{equation}
and where $B(0,1/2)\subseteq  L^{\infty}(0,p/2)$ denotes the closed ball centered at 0 and radius 1/2 in the Banach space $L^{\infty}(0,p/2)$.
It is evident from formula (\ref{Ep}) that $\Gl\mapsto E_p(\Gl;z)$ is an odd elliptic function with periods $p$ and $2\pi i$ and three poles at $\pi i$, $\pm\log(-z)$ modulo periods in each period cell with residues $2$, $-1$ and $-1$, respectively, provided these three points don't have congruent pairs modulo periods, i.e., $z\not=e^{pk/2}$, $k\in\bb{Z}$. 
The classical theory of elliptic functions gives the decomposition of $E_p(\Gl;x)$ in terms of the Weierstrass $\Gz$-functions (see, e.g., \cite[\S 14]{90AK})
\begin{equation}
  \label{Epzeta}
E_p(\Gl;z)=2\Gz(\Gl+i\pi)-\Gz(\Gl+i\pi+\log z)-\Gz(\Gl+i\pi-\log z),\quad x>0,
\end{equation}
where $\Gz(u)$ is the Weierstrass $\Gz$-function with periods $p$ and $2\pi i$. This formula shows that $E_{p}(\Gl;z)=0$, when $z=e^{pk}$, $k\in\bb{Z}$. If $z=e^{p/2}e^{kp}=c e^{kp}$, $k\in\bb{Z}$, then $E_p(\Gl;z)$ has only two poles at $\pi i$ and $\pi i+p/2$ in the rectangle of periods. The residues at these poles are 2 and $-2$, respectively. Therefore,
\begin{equation}
  \label{Epc}
  E_p(\Gl;c)=\frac{2\sqrt{m}K'(m)}{\pi}{\rm sn}\left(\frac{K'(m)\Gl}{\pi},m\right),\quad e^{\frac{2\pi K(m)}{K'(m)}}=c,
\end{equation}
where ${\rm sn}(u,m)$ is a Jacobi elliptic sine function, $K(m)$ is a complete elliptic integral of the first kind, and $K'(m)=K(1-m)$. Formula (\ref{Epc}) was obtained by matching the poles at $iK'$ and $2K+iK'$ of the Jacobi elliptic sine to the poles of $E_{p}(\Gl;c)$ and rescaling the residues $\pm1/\sqrt{m}$ of ${\rm sn}(u,m)$ at the poles to match the residues $\pm2$ of $E_{p}(\Gl;c)$. The parameter $m$ defined by the second equation in (\ref{Epc}) was chosen to match the ratio of the two periods of ${\rm sn}(u,m)$ and $E_p(\Gl;c)$.
\begin{Lem}
  \label{lem:Ep}
  When $z\not=e^{pk/2}$, $k\in\bb{Z}$, the elliptic function $\Gl\mapsto E_p(\Gl;z)$ has exactly three simple zeros at $\Gl=0$, $p/2$ and $p/2+\pi i$. Moreover, the Fourier coefficients of the sine series of $E_p(\Gl;z)$ regarded as an odd $p$-periodic function of $\Gl$ are
\begin{equation}
  \label{fc}
S_n(w)=\int_{0}^{p/2}E_{p}(\Gl;e^{w})\sin\left(\frac{2\pi n\Gl}{p}\right)d\Gl=2\pi \frac{\sin^{2}\left(\frac{\pi n w}{p}\right)}{\sinh\left(\frac{2\pi^{2}n}{p}\right)}.
\end{equation}
\end{Lem}
\begin{proof}
  Since $E_{p}(\Gl;z)$ has exactly 3 poles, it must have exactly 3 zeros, counting multiplicity, in the period rectangle. Moreover, by the Liouville theorem the sum of all three zeros, counting multiplicity, must be equal to the sum of poles, i.e., to $\pi i$, modulo periods. If $E_{p}(\Gl;z)$ has a single zero $\Gl_{0}$ of multiplicity 3, then $\Gl_{0}=0$, since $E_{p}(\Gl;z)$ is an odd function. But then, the sum of all zeros would be $0\not=\pi i$. If  
$E_{p}(\Gl;z)$ has a double-zero at $0$ and a simple zero at $\Gl_{1}\not=0$, then, by the Liouville theorem,  $\Gl_{1}=\pi i$ modulo periods, which is impossible, since $\pi i$ is a pole of $E_{p}$. If $0$ is a simple zero and $\Gl_{1}$ is a double zero, then $-\Gl_{1}$ must also be a double zero ($E_{p}$ is an odd function of $\Gl$). Therefore, $-\Gl_{1}$ must be congruent to $\Gl_{1}$, and
we must have $2\Gl_{1}=0$ modulo periods. But then the sum of all zeros $0+\Gl_{1}+\Gl_{1}=2\Gl_{1}=0\not=\pi i$ modulo periods. Hence, $E_{p}(\Gl;z)$ must have three distinct simple zeros: 0, $\Gl_{1}$ and $\Gl_{2}$. Since $E_{p}$ is odd, $-\Gl_{1}$ is also a zero of $E_{p}$. It is therefore must be congruent to either $\Gl_{1}$ or $\Gl_{2}$. In the latter case $\Gl_{1}+\Gl_{2}=0$ modulo periods, and the sum of all zeros will be $0+\Gl_{1}+\Gl_{2}=0\not=\pi i$ modulo periods. This contradiction shows that
$-\Gl_{1}$ must be congruent to $\Gl_{1}$. Therefore, $-\Gl_{2}$ must be congruent to $\Gl_{2}$. Hence we have $2\Gl_{1}=2\Gl_{2}=0$ and $\Gl_{1}+\Gl_{2}=\pi i$, modulo periods. Equation $2\Gl=0$ has only 3 nonzero solutions in the period rectangle: $\Gl=p/2$, $\Gl=i\pi$, and $\Gl=p/2+i\pi$. This gives $\Gl_{1}=p/2$ and $\Gl_{2}=\pi i+p/2$, since $\pi i$ is a pole.

Formula (\ref{fc}) is an immediate consequence of Theorem~\ref{th:Fourier} that says that
$\Psi(\Gl)=\hf\sin\left(\frac{2\pi n\Gl}{p}\right)$ corresponds to
$S(w)=\pi\frac{\sin^{2}\left(\frac{\pi wn}{p}\right)}{\sinh(2\pi^{2}n)}$, and (\ref{fc}) follows.
\end{proof}
\begin{figure}[t]
  \centering
  \hspace*{-2ex}\includegraphics[scale=0.25]{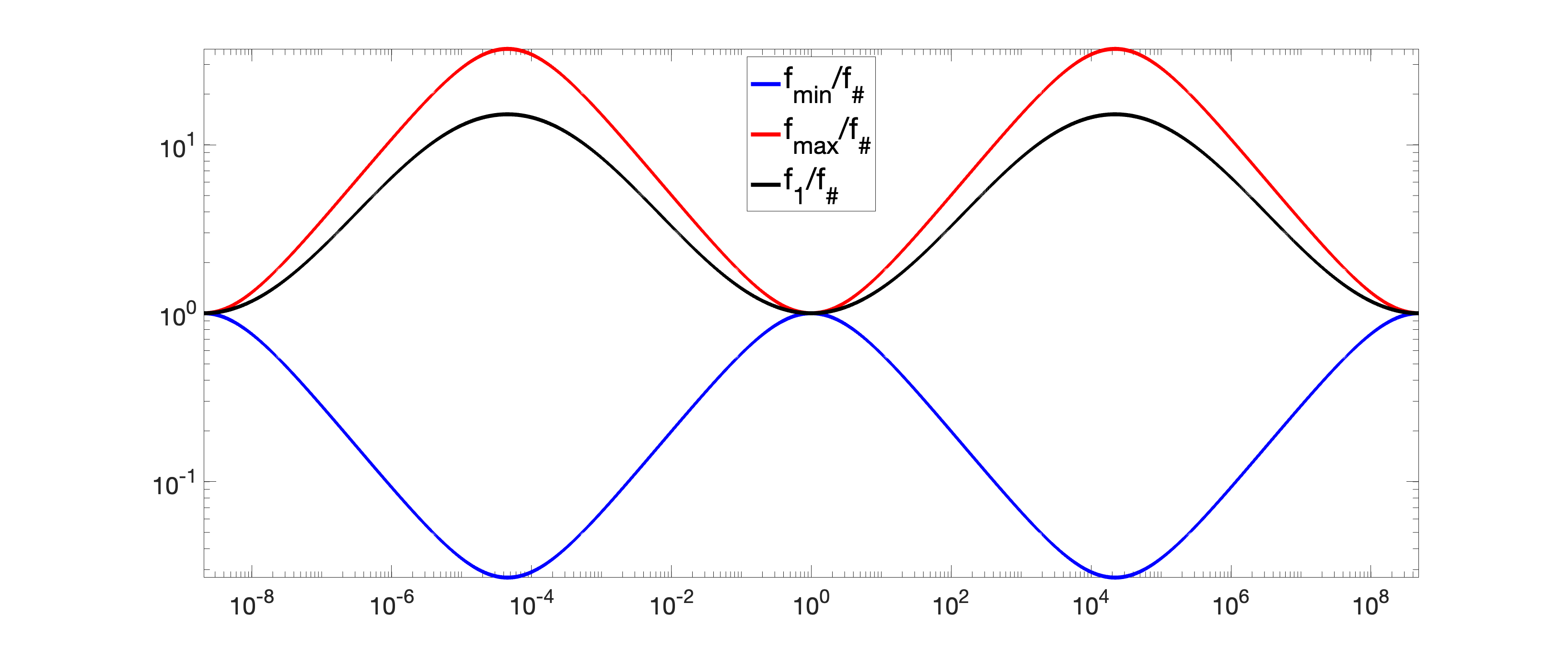}
  \caption[The extremal elements of $\CM_{c}$.]{Graphs of the minimal and the maximal elements of $\CM_{c}$ relative to $f_{\#}(x)=\sqrt{x}$ for $c=e^{10}$, corresponding to $p=20$. All other functions in $\CM_c$, such as $f_1$, given by (\ref{SeqNonGeomMolnarMeans}) with $n=1$, are sandwiched between $f_{\min}$ and $f_{\max}$. 
  }
  \label{fig:fminmax}
\end{figure}
\begin{Thm}[Order structure of the Moln\'ar class]
    \label{thm:order}
   The parametrization 
   \begin{equation}
       \label{Mparam}
       L^\infty(0,p/2)\supseteq B(0,1/2)\ni\Psi\mapsto f\in\CM_c
   \end{equation}
   is order-preserving, and therefore any $f\in\CM_{c}$ lies between the minimal and the maximal elements of $\CM_{c}$ (see Fig.~\ref{fig:fminmax}):
\begin{equation}
    \label{sandwich}
f_{\min}(x)=\sqrt{x}e^{-\frac{1}{2}\int_{0}^{p/2}E_{p}(\Gl;x)d\Gl}\le f(x)\le\sqrt{x}e^{\frac{1}{2}\int_{0}^{p/2}E_{p}(\Gl;x)d\Gl}=f_{\max}(x).
\end{equation}
Moreover,
\begin{gather}
\label{fmin}
f_{\min}(x)=\frac{\sqrt{x}}{\sqrt{m}+1}\left({\rm dn}\left(\frac{K'(m)\log x}{\pi},m\right)+\sqrt{m}{\rm cn}\left(\frac{K'(m)\log x}{\pi},m\right)\right)\in\CM_c,\\
\label{fmax}
f_{\max}(x)=\frac{x}{f_{\min}(x)}=\frac{(\sqrt{m}+1)\sqrt{x}}{{\rm dn}\left(\frac{K'(m)\log x}{\pi},m\right)+\sqrt{m}{\rm cn}\left(\frac{K'(m)\log x}{\pi},m\right)}\in\CM_c,
\end{gather}
where ${\rm cn}(u,m)$ and ${\rm dn}(u,m)$ are the Jacobi elliptic cosine and elliptic delta functions, and $m\in(0,1)$ is the unique solution of $4\pi K(m)/K'(m)=p=2\log c$.
We also have 
\begin{equation}
    \label{minmaxM}
    \inf_{f\in\CM}f(x)=f_{!}(x),\qquad\sup_{f\in\CM}f(x)=f_{\nabla}(x).
\end{equation}
where $f_{\nabla}$ and $f_{!}$ are defined in (\ref{fnabla}) and (\ref{f!}), respectively.
\end{Thm}
\begin{proof}
Lemma~\ref{lem:Ep} implies that $E_p(\Gl;x)$, $x>0$, $x\not=e^{kp}$, $k\in\bb{Z}$, is real and does not change sign on $\Gl\in(0,p/2)$. Moreover, $E_p(\Gl;x)>0$, $x>0$, $x\not=e^{kp}$, since $\sin(2\pi\Gl/p)>0$ on $(0,p/2)$ and formula (\ref{fc}) shows that
\[
  \int_{0}^{p/2}E_{p}(\Gl;x)\sin\left(\frac{2\pi\Gl}{p}\right)d\Gl=
  2\pi\frac{\sin^{2}\left(\frac{\pi \log x}{p}\right)}{\sinh\left(\frac{2\pi^{2}}{p}\right)}>0,
\]
when $x\not=e^{kp}$, $k\in\bb{Z}$. This implies that the parametrization (\ref{Mparam}) is order-preserving.

Using the expansion
\[
\hf=\sum_{n=1}^{\infty}\frac{2\sin\left(\frac{2\pi(2n+1)x}{p}\right)}{\pi(2n+1)},\quad x\in(0,p/2),
\]
corresponding to the odd $p$-periodic extension of $\Psi(\Gl)=1/2$ on $(0,p/2)$, we obtain
\[
  S_{*}(w):=\frac{1}{2}\int_{0}^{p/2}E_{p}(\Gl;e^{w})dx=\sum_{n=0}^{\infty}\frac{2}{\pi(2n+1)}
  \int_{0}^{p/2}E_{p}(\Gl;e^{w})\sin\left(\frac{2\pi(2n+1)\Gl}{p}\right)d\Gl.
\]
Using formula (\ref{fc}) we obtain
\[
S_{*}(w)=\sum_{n=0}^{\infty}\frac{2\left(1-\cos\left(\frac{2\pi(2n+1)w}{p}\right)\right)}{(2n+1) \sinh\left(\frac{2\pi^{2}(2n+1)}{p}\right)}.
\]
We observe that
\[
S_{*}'(w)=\frac{4\pi}{p}\sum_{n=0}^{\infty}\frac{\sin\left(\frac{2\pi(2n+1)w}{p}\right)}{ \sinh\left(\frac{2\pi^{2}(2n+1)}{p}\right)}.
\]
It remains to notice that the Fourier sine series coefficients of $S'_{*}(w)$ are exactly the Fourier sine series coefficients of $(1/2)E_{p}(w;c)$. We conclude, using formula (\ref{Epc}), that
\[
  S_{*}'(w)=\hf E_{p}(w;c)=\frac{\sqrt{m}K'(m)}{\pi}{\rm sn}\left(\frac{K'(m)w}{\pi},m\right),
\]
where $m\in(0,1)$ is the unique solution of $4\pi K(m)/K'(m)=p$. Therefore,
\[
S_{*}(w)=\int_{0}^{w}S_{*}'(v)dv=\sqrt{m}\int_{0}^{\frac{K'(m)w}{\pi}}{\rm sn}(u,m)du.
\]
Using the antiderivative formula given in \cite[p.~215]{90AK} we obtain
\[
S_{*}(w)=\log\left(\frac{\sqrt{m}+1}{{\rm dn}\left(\frac{K'(m)w}{\pi},m\right)+\sqrt{m}{\rm cn}\left(\frac{K'(m)w}{\pi},m\right)}\right),
\]
 This proves formulas (\ref{fmin}) and (\ref{fmax}).
In particular, $f_{\min}$ and $f_{\max}$, $m\in(0,1)$ are two other explicit infinite families of representing functions of non-geometric means in $\CM$. Their plots for $p=20$, relative to $f_{\#}$, given by (\ref{fgeom}) are shown in Fig.~\ref{fig:fminmax}.

Since
\[
\lim_{m\to 1^{-}}{\rm cn}(u,m)=\lim_{m\to 1^{-}}{\rm dn}(u,m)=\nth{\cosh u},
\]
we compute
\begin{equation}
  \label{liminmax}
  \lim_{m\to 1^-}f_{\max}(x)=\frac{x+1}{2}=f_{\nabla}(x),\qquad
\lim_{m\to 1^-}f_{\min}(x)=\frac{x}{f_{\nabla}(x)}=f_{!}(x).
\end{equation}
According to \cite[Theorem~4.5]{80KA}, the arithmetic mean is the maximum of all symmetric means, while the harmonic mean is the minimum. Since all Moln\'ar means are symmetric, and in view of (\ref{sandwich}) and (\ref{liminmax}), we have
\begin{gather}
    \label{envmin}
   f_{!}(x)\le\inf_{f\in\CM}f(x)=\inf_{c>1}\min_{f\in\CM_c}f(x)=\inf_{m\in (0,1)}f_{\min}(x)\le f_{!}(x),\\ f_{\nabla}(x)\ge\sup_{f\in\CM}f(x)=\sup_{c>1}\max_{f\in\CM_c}f(x)=\sup_{m\in(0,1)}f_{\max}(x)\ge f_{\nabla}(x).\label{envmax}
\end{gather}
Therefore, all inequalities in (\ref{envmin}) and (\ref{envmax}) are equalities, and (\ref{minmaxM}) is established. Fig.~\ref{fig:minmaxenv} illustrates this observation.
\end{proof}
\begin{figure}[t]
  \centering
  \includegraphics[scale=0.26]{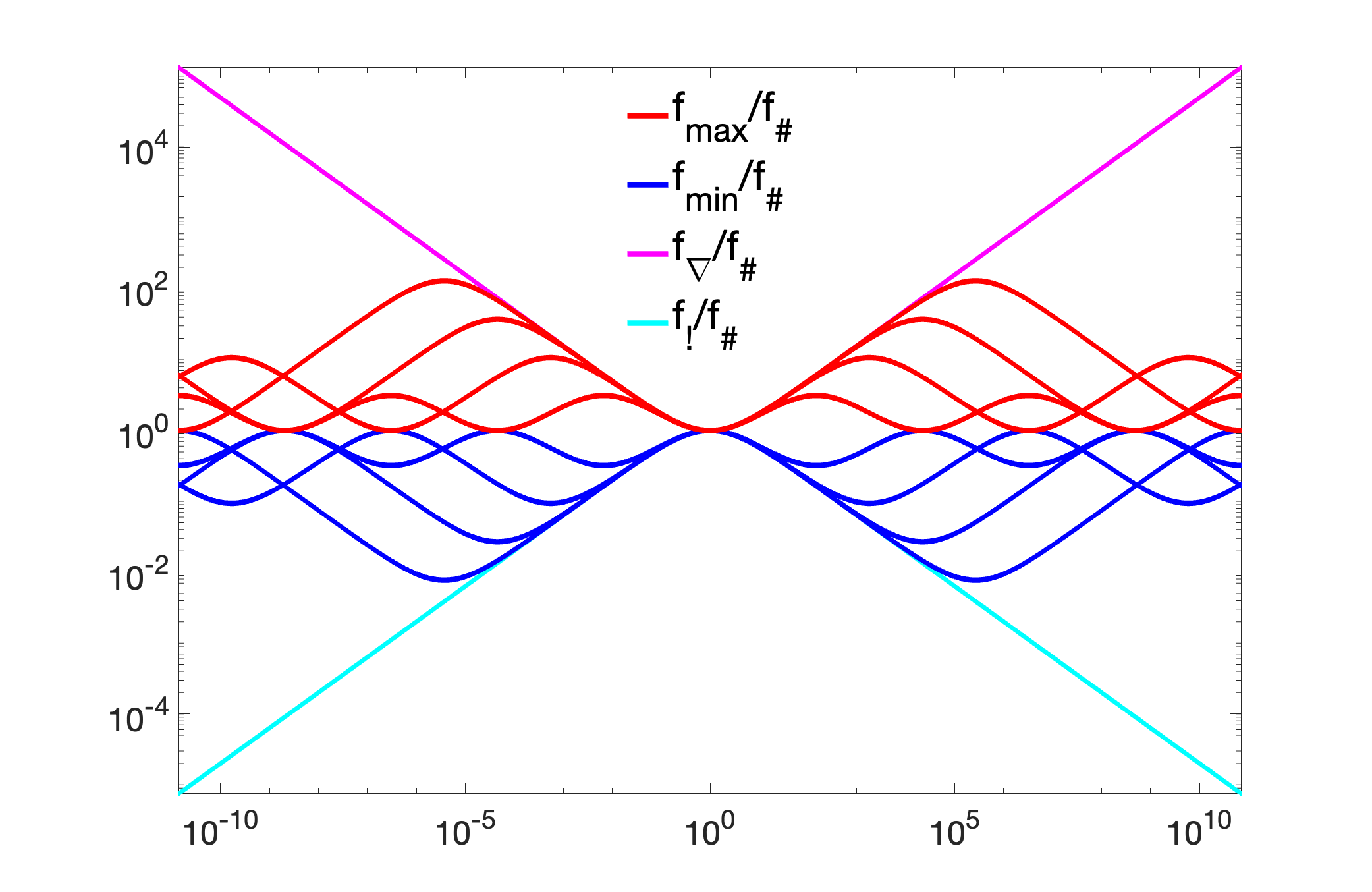}
  \caption[Envelope of extremals.]{$f_{\nabla}$ and $f_{!}$ are the limits of $f_{\max}$ and $f_{\min}$, respectively, as $c\to\infty$. The plots of $f_{\max}/f_{\#}$ and $f_{\min}/f_{\#}$ for $p=10,15,20,25$ are shown.}
  \label{fig:minmaxenv}
\end{figure}

Finally, we conclude the paper with a version of Lemma~\ref{lem:CharMolnarMeansRepresFunc} that uniquely characterizes the geometric mean.
\begin{Thm}[Characterization of the geometric mean]\label{th:CharGeoMean}
    A connection $\sigma$ with representing function $f$ (i.e., $\sigma=\sigma_f$) is a geometric mean $\#$ if and only if all of the following statements hold:
    \begin{itemize}
        \item[(i)] $f\in OM_+$;
        \item[(ii)] $f(1)=1$;
        \item[(iii)] $xf(1/x)=f(x),\text{ for }x>0$;
        \item[(iv)] there exist two logarithmically incommensurate positive scalars $c_1\not=1$, $c_2\not=1$ (i.e., $\frac{\log c_1}{\log c_2}\not\in\bb{Q}$) such that 
        $f\left(c_1^{2}x\right) =c_1f\left( x\right)$ and $f\left(c_2^{2}x\right) =c_2f\left( x\right)$, for $x>0$.
    \end{itemize}
\end{Thm}
\begin{proof}
Our analysis shows that a connection $\Gs$ satisfying all conditions of the theorem would correspond to an odd periodic function $\Psi(\Gl)$ that has two nonzero incommensurate periods $p_1=2\log c_1$ and $p_2=2\log c_2$. 
Therefore, $\Psi(\Gl)=0$ identically and $f(z)=\sqrt{z}$, corresponding to the geometric mean $\#$. 
\end{proof}

\section*{Acknowledgement}
YG is grateful to the National Science Foundation for support through grant DMS-2305832. GWM is grateful to the National Science Foundation for support through grant DMS-2107926. AW is grateful to the Simons Foundation for the support through grant MPS-TSM-00002799 and to the National Science Foundation for support through grant DMS-2410678. This paper was initiated at CIRM during the conference entitled ``Herglotz-Nevanlinna Functions and their Applications to Dispersive Systems and Composite Materials'', which was held on May 23-27, 2022. The authors wish to thank the organizers of that conference as well as CIRM for hosting us and for their hospitality. Both YG and AW would like thank the support they received from the Department of Mathematics at the University of Utah where a part of this paper was completed while they were on sabbatical from their respective universities.

\def\cprime{$'$} \ifx \cedla \undefined \let \cedla = \c \fi\ifx \cyr
  \undefined \let \cyr = \relax \fi\ifx \cprime \undefined \def \cprime
  {$\mathsurround=0pt '$}\fi\ifx \prime \undefined \def \prime {'}
  \fi
\bibliographystyle{abbrv}
\bibliography{BibMiltonWeltersCharMolnarClassPaper}
\end{document}